\documentclass[12pt]{article}
\usepackage{amssymb}
\usepackage{amsmath,amsthm}
\usepackage[latin1]{inputenc}
\usepackage{citesort}
\usepackage{color}
\usepackage{graphicx}
 \newtheorem{remark}{Remark}

 \newtheorem{theorem}[remark]{Theorem}
 
 \newtheorem{corollary}[remark]{Corollary}

   \newtheorem{example}[remark]{Example}

\title{Partitioning a graph into defensive $k$-alliances}

\author{ Ismael G.
Yero$^{1}$, Sergio Bermudo$^{2}$,\\ Juan A.
Rodr\'{\i}guez-Vel\'{a}zquez$^{1}$ and   Jos\'{e} M.
Sigarreta$^{3}$ \\
    \\
$^1${\small Department of Computer Engineering and Mathematics}\\
{\small Universitat Rovira i Virgili,}  {\small Av. Pa\"{\i}sos
Catalans 26, 43007 Tarragona, Spain.} \\{\small
ismael.gonzalez\@@urv.cat, juanalberto.rodriguez\@@urv.cat}
\\
$^2${\small Department of Economy, Quantitative Methods} {\small and
Economic History}\\ {\small Pablo de Olavide University,} {\small
Carretera de Utrera Km. 1, 41013-Sevilla, Spain} \\ {\small
sbernav\@@upo.es}\\
$^3${\small Faculty of Mathematics,} {\small Autonomous University
of Guerrero}
\\
{\small Carlos E. Adame 5, Col. La Garita, Acapulco, Guerrero,
M\'{e}xico }\\{\small josemariasigarretaalmira@hotmail.com}}

\date{}

\begin{document}

\maketitle

\begin{abstract}
A defensive $k$-alliance in a graph  is a set $S$ of vertices with
the property that every vertex in $S$ has at least $k$ more
neighbors in $S$ than it has outside of $S$. A defensive
$k$-alliance $S$ is called global if it forms a dominating set. In
this paper we  study the problem of partitioning the vertex set of a
graph into (global) defensive $k$-alliances. The (global) defensive
$k$-alliance partition number of a graph $\Gamma=(V,E)$,
($\psi_{k}^{gd}(\Gamma)$) $\psi_k^{d}(\Gamma)$,  is defined to be
the maximum number of sets in a partition of $V$ such that each set
is a (global) defensive $k$-alliance. We obtain tight bounds on
$\psi_k^{d}(\Gamma)$ and $\psi_{k}^{gd}(\Gamma)$ in terms of several
parameters of the graph including the order, size, maximum  and
minimum degree, the algebraic connectivity and the isoperimetric
number. Moreover, we study the close relationships that exist among
partitions of $\Gamma_1\times \Gamma_2$ into (global) defensive
$(k_1+k_2)$-alliances and partitions of $\Gamma_i$ into (global)
defensive $k_i$-alliances, $i\in \{1,2\}$.
\end{abstract}



\maketitle

\section{Introduction}

Since (defensive, offensive and dual) alliances in graph were first
introduced by P. Kristiansen, S. M. Hedetniemi and S. T. Hedetniemi
\cite{alliancesOne}, several authors have studied their mathematical
properties
\cite{cota,poweful,chellali,fava,FRS,GlobalalliancesOne,spectral,planar,yrs,GArs,kdaf,kdaf1,tesisShafique,SBF,line,offensive}.
 We are interested in a
generalization of defensive alliances, called $k$-alliances,
introduced by K. H. Shafique and R. D. Dutton in \cite{kdaf,kdaf1}.
We focus our attention in the problem of partitioning the vertex set
of a graph into defensive $k$-alliances. This problem  has been
previously studied by K. H. Shafique and R. D. Dutton
\cite{tesisShafique,porsi} and  the particular case $k=-1$ has been
studied by L. Eroh and R. Gera \cite{partitionTrees,partitionnumber}
and by T. W. Haynes and J. A. Lachniet \cite{partitionGrid}.

We begin by stating the terminology used. Throughout this article,
$\Gamma=(V,E)$ denotes a simple graph of order $|V|=n$ and size
$|E|=m$. We denote two adjacent vertices $u$ and $v$ by $u\sim v$,
 the degree of a vertex $v\in V$  by $\delta(v)$, the
minimum degree by $\delta$ and the maximum degree by $\Delta$. For a
nonempty set $X\subseteq V$, and a vertex $v\in V$,
 $N_X(v)$ denotes the set of neighbors  $v$ has in $X$:
$N_X(v):=\{u\in X: u\sim v\},$ and the degree of $v$ in $ X$ will be
denoted by $\delta_{X}(v)=|N_{X}(v)|.$  The subgraph induced by
$S\subset V$ will be denoted by  $\langle S\rangle $ and the
complement of the set $S$ in $V$ will be denoted by $\bar{S}$.

A nonempty set $S\subseteq V$ is a \emph{defensive  $k$-alliance} in
$\Gamma=(V,E)$,  $k\in \{-\Delta,\dots,\Delta\}$, if for every $
v\in S$,
\begin{equation}\label{cond-A-Defensiva} \delta _S(v)\ge \delta_{\bar{S}}(v)+k.\end{equation}

Notice that (\ref{cond-A-Defensiva}) is equivalent to
$$\delta (v)\ge 2\delta_{\bar{S}}(v)+k.$$

For example, if $k>1$, the star graph $K_{1,t}$ has no defensive
$k$-alliances and every set composed by two adjacent vertices in a
cubic graph is a defensive $(-1)$-alliance. For graphs having
defensive $k$-alliances, the \emph{defensive $k$-alliance number} of
$\Gamma$, denoted by $a_k^d(\Gamma)$, is defined as the minimum
cardinality of a defensive $k$-alliance in $\Gamma$. Notice that
$$a_{k+1}^d(\Gamma)\ge a_k^d(\Gamma).$$
 For the study of the
mathematical properties of $a_k^d(\Gamma)$ we cite \cite{yrs}.

A set $S\subset V$ is a  \emph{dominating
set}\label{conjuntodominante} in $\Gamma=(V,E)$ if for every vertex
$u\in \bar{S}$,  $\delta_S(u)>0$ (every vertex in $\bar{S}$ is
adjacent to at least one vertex in S). The \emph{domination number}
of $\Gamma$, denoted by $\gamma(\Gamma)$, is the minimum cardinality
of a dominating set in $\Gamma$.

A defensive $k$-alliance $S$ is called \emph{global} if it forms a
dominating set. For graphs having global defensive $k$-alliances,
the \emph{global defensive $k$-alliance number} of $\Gamma$, denoted
by $\gamma_{k}^{d}(\Gamma)$, is the minimum cardinality of a global
defensive $k$-alliance in $\Gamma$. Clearly,
$$
\gamma_{k+1}^d(\Gamma)\ge \gamma_{k}^d(\Gamma)\ge
\gamma(\Gamma)\quad {\rm and } \quad \gamma_k^d(\Gamma)\ge
a_k^d(\Gamma).$$
 For the study of the
mathematical properties of $\gamma_k^d(\Gamma)$ we cite \cite{GArs}.

The (\emph{global}) \emph{defensive $k$-alliance partition number}
of $\Gamma$, ($\psi_{k}^{gd}(\Gamma)$) $\psi_k^{d}(\Gamma)$, $k\in
\{-\Delta,...,\delta\}$, is defined to be the maximum number of sets
in a partition of $V (\Gamma)$ such that each set is a (global)
defensive $k$-alliance. Extreme cases are
$\psi_{-\Delta}^{d}(\Gamma)=n$, where each set composed of one
vertex is a defensive ($-\Delta$)-alliance, and
$\psi_{\delta}^{d}(\Gamma)=1$ for the case of a connected
$\delta$-regular graph where $V(\Gamma)$ is the only defensive
$\delta$-alliance. A graph $\Gamma$ is \emph{partitionable} into
(global) defensive $k$-alliances if ($\psi_{k}^{gd}(\Gamma)\ge 2$)
$\psi_{k}^{d}(\Gamma)\ge 2$. Hereafter we will say that
$\Pi_r(\Gamma)=\{V_1,V_2,...,V_r\}$ is a partition of $\Gamma$ into
$r$ (global) defensive $k$-alliances.

Notice that if every vertex of $\Gamma$ has even
 degree and $k$ is odd, $k=2l-1$,
 then
every (global) defensive $(2l-1)$-alliance in $\Gamma$ is a (global)
defensive $(2l)$-alliance and vice versa. Hence, in such a case,
$a^d_{2l-1}(\Gamma)=a^d_{2l}(\Gamma)$,
$\Gamma^d_{2l-1}(\Gamma)=\gamma^d_{2l}(\Gamma)$,
$\psi^d_{2l-1}(\Gamma)=\psi^d_{2l}(\Gamma)$ and
$\psi^{gd}_{2l-1}(\Gamma)=\psi^{gd}_{2l}(\Gamma)$.

Analogously, if every vertex of $\Gamma$ has odd
 degree and $k$ is even, $k=2l$,
 then
every defensive $(2l)$-alliance in $\Gamma$ is a defensive
$(2l+1)$-alliance and vice versa. Hence, in such a case,
$a^{d}_{2l}(\Gamma)=a^{d}_{2l+1}(\Gamma)$,
$\gamma^{d}_{2l}(\Gamma)=\gamma^{d}_{2l+1}(\Gamma)$,
$\psi^{d}_{2l}(\Gamma)=\psi^{d}_{2l+1}(\Gamma)$ and
$\psi^{gd}_{2l}(\Gamma)=\psi^{gd}_{2l+1}(\Gamma)$.

\section{Partitioning a graph into defensive $k$-alliances}

\begin{example}\label{Ej1} {\rm
Let $k$ and $r$ be integers such that $r>1$ and $r+k>0$  and let
${\cal H}$ be a family of graphs whose vertex set is
$V=\cup_{i=1}^r{V_i}$ where, for every $V_i$, $\langle
V_i\rangle\cong K_{r+k}$ and $\delta_{V_j}(v)=1$, for every $v\in
V_i$ and $j\ne i$. Notice that $\{V_1,V_2,...,V_r\}$ is a partition
of the graphs belonging to ${\cal H}$ into $r$ global defensive
$k$-alliances. A particular family of graphs included in ${\cal H}$
is $K_{r+k}\times K_r$.}
\end{example}

Hereafter, ${\cal H}$ will denote the family of graphs defined in the above example.

From the following relation between the defensive $k$-alliance
number, $a_{k}^{d}(\Gamma)$, and $\psi_{k}^{d}(\Gamma)$
 we obtain that lower bounds on $a_{k}^{d}(\Gamma)$ lead to upper
 bounds on
 $\psi_{k}^{d}(\Gamma)$:\begin{equation}\label{n-psi(no-global)}
a_{k}^{d}(\Gamma)\psi_{k}^{d}(\Gamma)\le n.
\end{equation}
For instance, it was shown in \cite{yrs} that
\begin{equation}\label{cotainf}a_{k}^{d}(\Gamma)\ge \left\lceil
 \frac{\delta+k+2}{2}\right\rceil.\end{equation}

An example of equality in the above bound is provided by the graphs belonging to the
family ${\cal H}$, for which we obtain $a_{k}^{d}(\Gamma)=r+k$.

By (\ref{n-psi(no-global)}) and (\ref{cotainf}) we obtain the
following bound,
$$\psi_{k}^{d}(\Gamma)\le \left\{\begin{array}{c}
                                          \left\lfloor\frac{2n}{\delta+k+2}\right\rfloor, \quad \delta+k\quad \mbox{\rm even}\\
                                          \\
                                            \left\lfloor\frac{2n}{\delta+k+3}\right\rfloor, \quad \delta+k\quad \mbox{\rm odd.}
                                         \end{array}\right.$$
This bound gives the exact value of $\psi_{k}^{d}(\Gamma)$, for
instance,  for every $\Gamma\in {\cal H}$, where
$\psi_{k}^{d}(\Gamma)=r$, and in the following cases:
$\psi_{-1}^{d}(K_4\times C_4)=5$, $\psi_{0}^{d}(K_3\times
C_4)=\psi_{-1}^{d}(K_2\times C_4)=4$ and $\psi_{1}^{d}(K_2\times
C_4)=2$.


Analogously, for global alliances we have
\begin{equation}\label{n-psi(global)}
\gamma_{k}^{d}(\Gamma)\psi_{k}^{gd}(\Gamma)\le n.
\end{equation}
One example of bounds on $\gamma_{k}^{d}(\Gamma)$ is the following, obtained
in \cite{GArs},
\begin{equation}\label{cotainfGDAN}
\gamma_{k}^{d}(\Gamma)\ge \left\lceil
\displaystyle\frac{n}{\left\lfloor\frac{\Delta-k}{2}\right\rfloor
+1}\right\rceil.
\end{equation}
For the graphs in ${\cal H}$, the above bound gives the exact value
$\gamma_{k}^{d}(\Gamma)=r+k $. Thus, the bound obtained by combining
(\ref{n-psi(global)}) and (\ref{cotainfGDAN}),
$$\psi_{k}^{gd}(\Gamma)\le
\left\lfloor\frac{\Delta-k}{2}\right\rfloor+1,$$
leads to the exact value of $\psi_{k}^{gd}(\Gamma)=r$ for every
$\Gamma\in {\cal H}$. Even so, this bound can be improved.

\newpage
\begin{theorem}
For every graph $\Gamma$ partitionable into global defensive $k$-alliances,
\begin{itemize}
\item[{\rm (i)}] $\psi_{k}^{gd}(\Gamma) \le \lfloor\frac{\sqrt{k^2+4n}-k}{2}\rfloor, $
\item[{\rm (ii)}] $\psi_{k}^{gd}(\Gamma) \le
\lfloor\frac{\delta-k+2}{2}\rfloor$.
\end{itemize}
\end{theorem}

\begin{proof}
Since, every $V_i\in \Pi_r(\Gamma)$ is a dominating set, we have
that for every $v\in V_i$, $\delta_{\overline{V_i}}(v)\ge r-1$.
Thus, the bounds are obtained as follow.
\begin{itemize}

\item[(i)] $|V_i|-1\ge \delta_{V_i}(v)\ge \delta_{\overline{V_i}}(v)+k\ge
r-1+k$, so $n=\sum_{i=1}^r|V_i|\ge r(r+k)$. By solving the
inequality $r^2+kr-n\le 0$ we obtain the result.

\item[(ii)]  Taking $v\in V_i$ as a vertex of minimum degree we obtain the result from
$\delta=\delta(v)\ge 2\delta_{\overline{V_i}}(v)+k\ge 2(r-1)+k$.
\end{itemize}

\end{proof}

The above bounds are attained, for instance, in the following cases:
$\psi_{-1}^{gd}(K_4\times C_4)=4$, $\psi_{0}^{gd}(K_3\times C_4)=3$,
$\psi_{1}^{gd}(K_2\times C_4)=2$  and  $\psi_1^{gd}(P)=2$, where $P$
denotes the Petersen graph.

\begin{remark}
For every $k\in \{1-\delta,...,\delta\}$, if
$\psi_{k}^{gd}(\Gamma)\ge 2$, then
$$\gamma_{k}^{d}(\Gamma)+\psi_{k}^{gd}(\Gamma)\le \frac{n+4}{2}.$$
\end{remark}

\begin{proof}
By (\ref{n-psi(global)}) we have
$\gamma_{k}^{d}(\Gamma)+\psi_{k}^{gd}(\Gamma)\le
\frac{n+\left(\psi_{k}^{gd}(\Gamma)\right)^2}{\psi_{k}^{gd}(\Gamma)}$.
On the other hand, if $k\in \{1-\delta,...,\delta\}$, then
$\gamma_{k}^{d}(\Gamma)\ge 2$. Moreover, if $\psi_{k}^{gd}(\Gamma)\ge 2$,  then  $ \gamma_{k}^{d}(\Gamma)\le \frac{n}{2}$. So,
$2\le \psi_{k}^{gd}(\Gamma)\le \frac{n}{\gamma_{k}^{d}(\Gamma)}\le \frac{n}{2}$. As a consequence, the result is obtained as
follow,
$$\max_{2\leq x \leq \frac{n}{\gamma_{k}^{d}(\Gamma)}}\left\{\frac{n+x^2}{x}\right\} =\max
\left\{\frac{n+4}{2},\frac{n+(\gamma_{k}^{d}(\Gamma))^{2}}{\gamma_{k}^{d}(\Gamma)}\right\}=
\frac{n+4}{2}.$$
\end{proof}
Example of equality in above bound is $\gamma_{-1}^{d}(C_4\times
K_2)+\psi_{-1}^{gd}(C_4\times K_2) = 6.$

\begin{theorem}\label{cortearistas}
Let $C_{(r,k)}^{gd} (\Gamma)$ be the minimum number of edges having
its endpoints in different sets of a partition of $\Gamma$ into
$r\ge 2$ global defensive $k$-alliances. Then
\begin{itemize}
\item[{\rm (i)}] $C_{(r,k)}^{gd} (\Gamma)\ge
\frac{1}{2}r(r-1)\gamma_k^d(\Gamma)$,

\item[{\rm (ii)}]  $C_{(r,k)}^{gd} (\Gamma)\ge
\frac{1}{2}r(r-1)(r+k)$,

\item[{\rm (iii)}]  $C_{(r,k)}^{gd} (\Gamma) \le \frac{2m-nk}{4}.$

\item[{\rm (iv)}] $C_{(r,k)}^{gd} (\Gamma) =
\frac{1}{2}r(r-1)\gamma_k^d(\Gamma)=\frac{1}{2}r(r-1)(r+k)=\frac{2m-nk}{4}$
if and only if $\Gamma \in {\cal H}$.

\end{itemize}
\end{theorem}

\begin{proof}
Let $x=\displaystyle\min_{V_i\in \Pi_r(\Gamma)}|V_i|$. From the fact
that every set of $\Pi_r(\Gamma)$ is a dominating set, we obtain
that the number of edges adjacent to $v\in V_i$ with one endpoint in
$\cup_{j=i+1}^r V_j$ is bounded by $\sum_{
j=i+1}^r\delta_{V_j}(v)\ge r-i$. Therefore,
\begin{equation}\label{boundx}
C_{(r,k)}^{gd} (\Gamma)\ge \sum_{i=1}^{r-1}(r-i)|V_{i}|
\ge x\sum_{i=1}^{r-1}(r-i)=\frac{x}{2}r(r-1).
\end{equation}
           Since  every $V_i\in \Pi_r(\Gamma)$ is a
global defensive $k$-alliance, we have $x\ge r+k$ and $x\ge
\gamma_k^d(\Gamma)$, as a consequence,  ({\it i}) and ({\it ii})
follow.

Proof of ({\it iii}). In order to obtain the upper bound we note
that the number of edges in $\Gamma$ with one endpoint in $V_i$ and
the other endpoint in $V_j$ is  $C(V_i,V_j)=\displaystyle\sum_{v\in
V_i}\delta_{{V_j}}(v)=\sum_{v\in V_j}\delta_{{V_i}}(v)$. Hence,
\begin{align*} 2m=\sum_{i=1}^r\sum_{v\in V_i}\delta(v)&\ge
2\sum_{i=1}^r\sum_{v\in
V_i}\delta_{\overline{V_i}}(v)+k\sum_{i=1}^r|V_i| \\
&=2\sum_{i=1}^r \sum_{v\in V_i} \sum_{j=1, j\ne i}^r\delta_{{V_j}}(v)+kn\\
&=2\sum_{i=1}^r \sum_{j=1, j\ne i}^r\sum_{v\in V_i}\delta_{{V_j}}(v)+kn\\
&=2\sum_{i=1}^r \sum_{j=1, j\ne i}^rC(V_i,V_j)+nk \\
&=4C_{(r,k)}^{gd} (\Gamma)+nk.
\end{align*}

Proof of ({\it iv}). $(\Rightarrow)$ If for some $V_i\in \Pi_r(\Gamma)$ there exists $v\in V_i$ such that $\delta_{V_i}(v)>\delta_{\overline{V_i}}(v)+k$, then, by analogy to the proof of (iii) we obtain $C_{(r,k)}^{gd} (\Gamma) <
\frac{2m-nk}{4}$. Therefore, if $C_{(r,k)}^{gd} (\Gamma) =\frac{2m-nk}{4}$, then for every $V_i\in
\Pi_r(\Gamma)$, and for every $v\in V_i$, we have 
\begin{equation}\label{sonfront}
\delta_{V_i}(v)=\delta_{\overline{V_i}}(v)+k.
\end{equation}
Moreover, if for some $V_i\in \Pi_r(\Gamma)$ there exists $v\in V_i$ such that $\displaystyle\sum_{j\ne i}\delta_{V_i}(v)>r-1$, then, by analogy to the proof of (i) and (ii) we obtain 
$C_{(r,k)}^{gd} (\Gamma) >\frac{1}{2}r(r-1)\gamma_k^d(\Gamma)$ and $C_{(r,k)}^{gd} (\Gamma) >\frac{1}{2}r(r-1)(r+k)$. Therefore, if 
$C_{(r,k)}^{gd} (\Gamma) =\frac{1}{2}r(r-1)\gamma_k^d(\Gamma)=\frac{1}{2}r(r-1)(r+k)$, then for every $V_i\in
\Pi_r(\Gamma)$, and for every $v\in V_i$, we have 
\begin{equation}\label{dominantes}
\delta_{\overline{V_i}}(v)=\displaystyle\sum_{j\ne i}\delta_{V_i}(v)=r-1.
\end{equation}
So, by (\ref{sonfront}) and (\ref{dominantes}) we obtain that for every  $V_i\in
\Pi_r(\Gamma)$,  $\langle V_i\rangle$  is regular of
degree $r+k-1$.  Thus, $\Gamma$
is a regular graph of degree $2(r-1)+k$ and, by
$\frac{1}{2}r(r-1)\gamma_k^d(\Gamma)=\frac{1}{2}r(r-1)(r+k)=\frac{2m-nk}{4}$
 we have $n(\Gamma)=r(r+k)$ and $\gamma_k^d(\Gamma)=r+k$.  Hence,
$|V_i|=r+k$, so $\langle V_i\rangle\cong K_{r+k}$. Moreover,  as
every $V_j\in \Pi_r(\Gamma)$ is a dominating set, by (\ref{dominantes}) we have
$\delta_{V_j}(v)=1$, for every $v\in V_i$, $i\ne j$. Therefore,
$\Gamma \in {\cal H}$. $(\Leftarrow)$ The result is immediate.
\end{proof}




By (\ref{boundx}) and Theorem \ref{cortearistas} (iii) we obtain the
following result.

\begin{corollary} \label{coroEqualcardinalityone}
For  every graph  $ \Gamma$ partitionable into $r$ global defensive
$k$-alliances of equal cardinality,  $r\le \frac{2(m+n)-kn}{2n}$.
\end{corollary}

A family of graphs that achieve equality for Corollary
\ref{coroEqualcardinalityone} is the family ${\cal H}$ defined in
Example \ref{Ej1}.




By Theorem \ref{cortearistas}  and (\ref{cotainf})  we obtain the following two
necessary conditions for the existence of a partition of a graph into
$r$ global defensive $k$-alliances.

\begin{corollary}
If for a graph $\Gamma$, $k>\frac{2m-r(r-1)(\delta+2)}{n+r(r-1)}$ or $k>\frac{2(m-r^2(r-1))}{n+2r(r-1)}$,
the $\Gamma$ cannot be partitioned into $r$ global defensive
$k$-alliances.
\end{corollary}

By the above corollary we conclude, for instance, that the 3-cube
graph cannot be partitioned into $r>2$ global defensive
$k$-alliances.


\begin{remark}
The size of the subgraph induced by a set belonging to a partition
of $\Gamma$ into $r$ global defensive $k$-alliances is bounded below
by $\frac{1}{2}\gamma_k^{d}(\Gamma)(r+k-1)$.
\end{remark}

\begin{proof}
The result follows from the fact that for every $V_i\in
\Pi_r(\Gamma)$, $\displaystyle\sum_{v\in V_i}\delta_{V_i}(v)\ge
((r-1)+k)|V_i|\ge (r-1+k)\gamma_k^{d}(\Gamma).$
\end{proof}
The above bound is tight as we can check by taking $\Gamma\in {\cal
H}$.

\subsection{Isoperimetric number,  bisection and  $k$-alliances}
The isoperimetric number of $\Gamma$ is defined as
$$ {\bf i}(\Gamma):= \displaystyle\min_{S\subset V(\Gamma): |S|\le \frac{n}{2}} \left\{\frac{\sum_{v\in S}\delta_{\overline{S}}(v)}{|S|}\right\}.$$
As a consequence of Theorem \ref{cortearistas} (iii) we obtain the
following result.

\begin{corollary}
If there exists a partition $\Pi_r$ of $\Gamma$ into $r\ge2$ global
defensive $k$-alliances such that, for every $V_i\in \Pi_r$,
$|V_i|\le \frac{n}{2}$, then
 $${\bf i}(\Gamma)\le \frac{2m-nk}{2n}.$$
\end{corollary}

\begin{proof}
For every $V_i\in \Pi_r$ we have $|V_i|{\bf i}(\Gamma)\le
\displaystyle\sum_{v\in V_i}\delta_{\overline{V_i}}(v)=\sum_{v\in
V_i}\sum_{j=1,j\ne i}^r\delta_{V_j}(v)$. Hence,
$$n{\bf i}(\Gamma)={\bf i}(\Gamma)\sum_{i=1}^r|V_i|\le  \sum_{i=1}^r\sum_{v\in
V_i}\sum_{j=1,j\ne i}^r\delta_{V_j}(v)=2C_{(r,k)}^{gd}(\Gamma)\le
\frac{2m-nk}{2}.$$
\end{proof}
Example of equality in above bound is the graph $\Gamma=C_3\times
C_3$ for $k=0$. That is, $C_3\times C_3$ can be partitioned into
$r=3$ global defensive $0$-alliances of cardinality $3$, moreover,
${\bf i}(C_3\times C_3)=2$. Other example is the $3$-cube graph
$\Gamma=C_4\times K_2$, for $k=1$. In this case each copy of the
cycle $C_4$ is a global defensive $1$-alliance and  ${\bf
i}(C_4\times K_2)=1$.

Notice that if 
${\bf i}(\Gamma)> \frac{2m-nk}{2n},$ 
then $\Gamma$ cannot be partitioned into $r\ge 2$ global defensive
$k$-alliances  with the condition that the cardinality of every set
in the partition is at most $\frac{n}{2}.$  One example of this is the graph
$\Gamma=C_3\times C_3$ for $k\ge1$.

\begin{theorem} For any graph $\Gamma$,

\begin{itemize}
\item[{\rm (i)}]
if  $\Gamma$ is partitionable into global defensive $k$-alliances,
then
 $$\psi_k^{gd}(\Gamma) \le  \Delta+1-{\bf i}(\Gamma)-k,$$
 \item[{\rm (i)}] if  $\Gamma$ is partitionable into  defensive $k$-alliances,
then  $$a_k^d(\Gamma)\ge {\bf i}(\Gamma)+k+1.$$
\end{itemize}
\end{theorem}

\begin{proof}
\mbox{ }

\begin{itemize}
\item[{\rm (i)}]
Let $\Pi_r(\Gamma)$ be a partition of $\Gamma$ into  $r\ge2$ global
defensive $k$-alliances. Then, there exists $V_i\in \Pi_r(\Gamma)$
such that $|V_i|\le \frac{n}{2}$. Hence,
 $|V_i|{\bf
i}(\Gamma)\le \displaystyle\sum_{v\in V_i}\delta_{\overline{V_i}}(v)
\le \sum_{v\in V_i}(\delta_{V_i}(v)-k)\le  \sum_{v\in
V_i}(\delta(v)-r+1-k)\le |V_i|(\Delta-r+1-k)$. Thus,
$r\le \Delta+1-{\bf i}(\Gamma)-k.$

\item[{\rm (ii)}] If $\psi_k^d(\Gamma)\ge 2$, then there exists a
defensive $k$-alliance $S$ such that $|S|\le \frac{n}{2}$.
Therefore, $\displaystyle |S|{\rm i}(\Gamma)\le \sum_{v\in
S}\delta_{\overline{S}}(v)\le \sum_{v\in S}(\delta_S(v)-k)\le
|S|(|S|-1)-k|S|$. Thus, the result follows.
\end{itemize}
\end{proof}




The following relation between the algebraic connectivity and the
isoperimetric number  of a graph was shown by Mohar in \cite{Mohar}:
${\bf i}(\Gamma)\ge \frac{\mu}{2}.$
\begin{corollary}\label{coromu} For any graph $\Gamma$,
\mbox{ }

\begin{itemize}
\item[{\rm (i)}] if  $\Gamma$ is partitionable into  global defensive
$k$-alliances, then $$\psi_k^{gd}(\Gamma) \le
\left\lfloor\Delta+1-\frac{\mu}{2}-k\right\rfloor,$$

\item[{\rm (ii)}] if $\Gamma$ is partitionable into defensive $k$-alliances,  then  $$a_k^d(\Gamma)\ge
\left\lceil\frac{\mu+2(k+1)}{2}\right\rceil.$$
\end{itemize}
\end{corollary}
Example of equality in above bounds is the graph $\Gamma=C_3\times
C_3$ for $k=0$,  in this case $\mu=3$.

From above corollary, we emphasize that if  $\mu > 2(\Delta-1-k)$,
then $\Gamma$ cannot be partitioned into global defensive
$k$-alliances. For instance, we conclude that $\Gamma=C_3\times C_3$
cannot be partitioned into global defensive $k$-alliances for $k>1$.
Moreover, by Corollary \ref{coromu} (ii) we conclude, if
$a_k^d(\Gamma)< \left\lceil\frac{\mu+2(k+1)}{2}\right\rceil$, then
$\Gamma$ cannot be partitioned into defensive $k$-alliances.

A {\it bisection} of $\Gamma$ is a 2-partition $\{X,Y\}$ of the
vertex set $V(\Gamma)$ in which $\left| X\right| =\left| Y\right| $
or $\left| X\right| =\left| Y\right| +1.$ The bisection problem is
to find a bisection for which $\sum_{v\in X}\delta_Y(v)$ is as small
as possible.  The \mbox{\it bipartition width}, $bw(\Gamma)$, is
defined as $$ bw(\Gamma):=\displaystyle\min_{X\subset V(\Gamma),
\left| X \right| =\left \lfloor\frac{n}{2} \right\rfloor } \left\{
\sum_{v\in X}\delta_{\overline{X}}(v) \right\}.$$

It was shown by  Merris \cite{Merris} and Mohar
 \cite{Mohar} that
$$  bw(\Gamma) \geq
     \left\lbrace \begin{array}{llll}
                \left \lceil  \frac{n\mu }{4} \right \rceil & \mbox{\it if  $n$ is even;}
                    \\

                    \\
                \left \lceil  \frac{(n^2-1)\mu }{4n} \right \rceil & \mbox{\it if  $n$ is odd.}
                \end{array}
\right.   $$

We are interested in the bisection of a graph into  global defensive
$k$-alliances, i.e., the bisection  $\{X,Y\}$ of $V$ such that $X$
and $Y$ are global defensive $k$-alliances. An example of bisection
into global defensive (t-1)-alliances is obtained for the family of
hypercube graphs $Q_{t+1}=Q_t\times K_2$, by taking $\{X,Y\}$ such
that $\langle X\rangle\cong Q_t\cong\langle Y\rangle$.

By Theorem \ref{cortearistas} (iii) and the above bound we obtain
the following result.

\begin{corollary}\label{nobisection}
If $\left\lfloor\frac{2m-nk}{4}\right\rfloor < \left\lceil
\frac{n\mu }{4}\right\rceil$, for $n$ even, or
$\left\lfloor\frac{2m-nk}{4}\right\rfloor < \left\lceil
\frac{(n^2-1)\mu }{4n}\right\rceil$, for $n$ odd, then $\Gamma$
cannot be   bisectioned into global defensive $k$-alliances.
\end{corollary}

For example, according to Corollary \ref{nobisection} we can
conclude that, for $k>0$, the graph $C_3\times C_3$ cannot be
bisectioned into global defensive $k$-alliances.




\section{Partitioning $\Gamma_1\times \Gamma_2$ into (global) defensive $k$-alliances}

In  Subsection \ref{partitioningNoglobal} we will discuss the  close
relationships that exist among  $\psi_{k_1+k_2}^{d}(\Gamma_1\times
\Gamma_2)$ and $ \psi_{k_i}^{d}(\Gamma_i)$, $i\in \{1,2\}$.
Obviously, we begin with the study of the relationship between
$a_{k_1+k_2}^{d}(\Gamma_1\times \Gamma_2)$ and $a_{k_i}^{d}
(\Gamma_i)$, $i\in \{1,2\}$. The case of global alliances will be
studied in Subsection \ref{partitioningGlobal}.

\subsection{Partitioning $\Gamma_1\times \Gamma_2$ into defensive
$k$-alliances}\label{partitioningNoglobal}

\begin{theorem}\label{CartesianNG}  For any graphs $\Gamma_1$ and $\Gamma_2$,
\begin{itemize}
\item[{\rm (i)}] if $\Gamma_i$ contains a defensive $k_i$-alliance,  $i\in
\{1,2\}$, then $\Gamma_1\times \Gamma_2$ contains a defensive
$(k_1+k_2)$-alliance and $$a_{k_1+k_2}^{d}(\Gamma_1\times
\Gamma_2)\le a_{k_1}^{d} (\Gamma_1)a_{k_2}^{d} (\Gamma_2),$$
\item[{\rm (ii)}] if there exists a partition of $\Gamma_i$ into defensive $k_i$-alliances,  $i\in
\{1,2\}$, then there exists a partition of $\Gamma_1\times \Gamma_2$
into defensive $(k_1+k_2)$-alliances and
$$\psi_{k_1+k_2}^{d}(\Gamma_1\times \Gamma_2) \ge
\psi_{k_1}^{d}(\Gamma_1)\psi_{k_2}^{d}(\Gamma_2).$$
 \end{itemize}
\end{theorem}

\begin{proof}
Let  $S_i$ be a defensive $k_i$-alliance  in $\Gamma_i$, $i\in
\{1,2\}$, and let  $X=S_1\times S_2$. Then for every $x=(u,v)\in X$,
\begin{align*}\delta_{X}(x)&=\delta_{S_1}(u)+\delta_{S_2}(v)\\ &
\ge
\left(\delta_{\bar{S_1}}(u)+k_1\right)+\left(\delta_{\bar{S_2}}(v)+k_2\right)
\\&=\delta_{\bar{X}}(x)+k_1+k_2.\end{align*} Thus, $X$ is a
defensive $(k_1+k_2)$-alliance in $\Gamma_1\times \Gamma_2$ and, as
a consequence,  (i) follows. Moreover, we conclude that every
partition
$$\Pi_{r_i}(\Gamma_i)=\{S_{1}^{(i)},S_{2}^{(i)},...,S_{r_i}^{(i)}\}$$
 of $\Gamma_i$ into $r_i$  defensive $k_i$-alliances
 induces a partition  of $\Gamma_1\times \Gamma_2$
into $r_1r_2$ defensive $(k_1+k_2)$-alliances:
$$\Pi_{r_1r_2}(\Gamma_1\times \Gamma_2)=\left\{
\begin{array}{ccc}
  S_{1}^{(1)}\times S_{1}^{(2)} & \cdots & S_{1}^{(1)}\times S_{r_2}^{(2)} \\
  S_{2}^{(1)}\times S_{1}^{(2)} & \cdots & S_{2}^{(1)}\times S_{r_2}^{(2)} \\
 \vdots & \vdots & \vdots \\
  S_{r_1}^{(1)}\times S_{1}^{(2)} & \cdots & S_{r_1}^{(1)}\times S_{r_2}^{(2)}
\end{array}\right\} .
$$
Therefore, (ii) follows.
\end{proof}

In the particular case of the Petersen graph, $P$, and the $3$-cube
graph, $Q_3$, we have $a_{-2}^{d}(P\times Q_3) = 4=
a_{-1}^d(P)a_{-1}^d(Q_3)$, $\psi_{-2}^{d}(P\times Q_3) =20=
\psi_{-1}^{d}(P)\psi_{-1}^{d}(Q_3)$ and  $16=a_{2}^{d}(P\times Q_3)
< a_{1}^d(P)a_{1}^d(Q_3)=20$, $5=\psi_{2}^{d}(P\times Q_3) >
\psi_{1}^{d}(P)\psi_{1}^{d}(Q_3)=4$.

An example where we cannot apply Theorem \ref{CartesianNG} (i) is the book graph
$\Gamma_1\times \Gamma_2=K_{1,4}\times K_2$, for $k_1=2$ and
$k_2=0$; the star graph $\Gamma_1=K_{1,4}$ does not contain
defensive $2$-alliances, although $\Gamma_1\times \Gamma_2$ contains
some of them and $a_2^{d}(\Gamma_1\times \Gamma_2)=8$.

We note that from Theorem \ref{CartesianNG} we obtain $
a_{2k}^{d}(\Gamma_1\times \Gamma_2)\le a_{k }^{d} (\Gamma_1)a_{k
}^{d} (\Gamma_2)$ and $\psi_{2k}^{d}(\Gamma_1\times \Gamma_2)\ge
\psi_{k}^{d}(\Gamma_1)\psi_{k}^{d}( \Gamma_2).$ Another interesting
consequence of Theorem \ref{CartesianNG} is the following.

\begin{corollary} Let $\Gamma_1$ and $\Gamma_2$ be two graphs of order $n_1$ and $n_2$ and  maximum
 degree $\Delta_1$ and $\Delta_2$, respectively. Let $s\in
 \mathbb{Z}$ such that
 $\max\{\Delta_1,\Delta_2\}\le s\le \Delta_1+\Delta_2+k$. Then

\begin{itemize}

\item[{\rm (i)}] $a_{_{k-s}}^{d}(\Gamma_1\times \Gamma_2)\le \min \{a_{k}^{d}
(\Gamma_1),a_{k}^{d} (\Gamma_2)\}$,

\item[{\rm (ii)}] $\psi_{k-s}^{d}(\Gamma_1\times
\Gamma_2) \ge \max
\{n_2\psi_k^{d}(\Gamma_1),n_1\psi_k^{d}(\Gamma_2)\}$.

 \end{itemize}
\end{corollary}

As example of equalities we take $\Gamma_1=P$, $\Gamma_2=Q_3$, $k=1$
and $s=3$. In such a case,  $4=a_{-2}^{d}(P\times Q_3)= \min
\{a_{1}^{d} (P),a_{1}^{d} (Q_3)\}=\min\{5,4\}$ and
$20=\psi_{-2}^{d}(P\times Q_3) = \max
\{8\psi_1^{d}(P),10\psi_1^{d}(Q_3)\}=\max\{16,20\}$.

\subsection{Partitioning $\Gamma_1\times \Gamma_2$ into global defensive
$k$-alliances}\label{partitioningGlobal}

\begin{theorem}\label{ThGlobCartesian} Let $\Pi_{r_i}(\Gamma_i)$ be a partition of a graph $\Gamma_i$, of order
$n_i$, into $r_i\ge 1$ global defensive $k_i$-alliances, $i\in
\{1,2\}$, $r_1\le r_2$. Let $x_i=\displaystyle\min_{X\in
\Pi_{r_i}(\Gamma_i)}\{|X|\}$. Then,

\begin{itemize}

\item[{\rm (i)}] $\displaystyle\gamma_{_{k_1+k_2}}^{d}(\Gamma_1\times \Gamma_2)\le
\min\left\{x_1n_2,x_2n_1 \right\},$
\item[{\rm (ii)}] $\displaystyle\psi_{k_1+k_2}^{gd}(\Gamma_1\times
\Gamma_2) \ge \max\left\{\psi_{k_1}^{gd}(\Gamma_1),\psi_{k_2}^{gd}(\Gamma_2)\right\}$.
 \end{itemize}
\end{theorem}

\begin{proof}
From the procedure showed in the proof of Theorem \ref{CartesianNG}
we obtain that for every $S_j^{(1)}\in \Pi_{r_1}(\Gamma_1)$ and every $S_l^{(2)}\in
\Pi_{r_2}(\Gamma_2)$, the sets $M_j=S_j^{(1)}\times V_2$  and $N_l= V_1\times S_l^{(2)}$ are defensive $(k_1+k_2)$-alliances in $\Gamma_1\times \Gamma_2$. Moreover
$M_j$ and $N_l$ are dominating sets in $\Gamma_1\times \Gamma_2$. Thus,
by taking $S_{j}^{(1)}$ and $S_{l}^{(2)}$ of cardinality $x_1$ and
$x_2$, respectively, we obtain $|M_j|= x_1n_2$ and $|N_l|= x_2n_1$,
so (i) follows. Moreover, as $\{M_1,...,M_{r_1}\}$ and
$\{N_1,...,N_{r_2}\}$ are partitions of $\Gamma_1\times \Gamma_2$
into global defensive $(k_1+k_2)$-alliances, (ii) follows.
\end{proof}


\begin{corollary}\label{coroProduct}
If  $\Gamma_i$ is a graph  of order $n_i$ such that
$\psi_{k_i}^{gd}(\Gamma_i)\ge 1$, $i\in \{1,2\}$, then
$$\displaystyle\gamma_{_{k_1+k_2}}^{d}(\Gamma_1\times \Gamma_2)\le
\frac{n_1n_2}{\max_{i\in \{1,2\}}\left\{\psi_{k_i}^{gd}(\Gamma_i)
\right\}}.$$
\end{corollary}

\begin{theorem}\label{ThGanmaN2} If  $\Gamma_1$ contains a global defensive
$k_1$-alliance, then for every $k_2\in \{-\Delta_2,...,\delta_2\}$, $\Gamma_1\times \Gamma_2$ contains a global
defensive $(k_1+k_2)$-alliance and $
\gamma_{_{k_1+k_2}}^{d}(\Gamma_1\times \Gamma_2)\le
\gamma_{_{k_1}}^d(\Gamma_1)n_2.$
\end{theorem}

\begin{proof} Following a similar procedure used in the proof of
Theorem \ref{ThGlobCartesian} (i) we deduce the result.
\end{proof}

For the graph $\Gamma_1\times \Gamma_2=C_4\times Q_3$, by taking
$k_1=0$ and $k_2=1$, we obtain equalities in Theorem
\ref{ThGlobCartesian}, Corollary \ref{coroProduct} and Theorem
\ref{ThGanmaN2}.


\section*{Acknowledgments}
This work was partly supported by the Spanish Ministry of Science
and Innovation through projects TSI2007-65406-C03-01 ``E-AEGIS",
CONSOLIDER INGENIO 2010 CSD2007-0004 "ARES"  and  MTM2009-09501, by
the Rovira i Virgili University through project 2006AIRE-09 and by
the Junta de Andaluc\'ia, ref. FQM-260 and ref. P06-FQM-02225.

\end{document}